\documentclass[a4paper, 11pt]{amsart}

\usepackage{amsmath,amssymb,enumerate,comment,amsthm,mathrsfs,mathtools}
 \usepackage{amsfonts}
 
 \usepackage{microtype}

\usepackage{amscd}
\usepackage{array}

\newtheorem{Thm}{Theorem}[section]

\newtheorem{Lem}[Thm]{Lemma}

\newtheorem{Cor}[Thm]{Corollary}

\theoremstyle{definition}

\newtheorem{Rem}[Thm]{Remark}

\newcommand{\Cs}{C$^\ast$}

\newcommand{\rca}[1]{\mathop{\rtimes _{#1, {\mathrm r}}}}

\newcommand{\Sub}[1]{\mathop{\mathrm{Sub}}({#1})}
\newcommand{\Prob}{\mathrm{Prob}}

\newcommand{\IF}{\mathbb F}

\newcommand{\IN}{\mathbb N}
\newcommand{\IR}{\mathbb R}

\newcommand{\IZ}{\mathbb Z}
\newcommand{\cB}{\mathcal B}

\newcommand{\cH}{\mathcal H}

\newcommand{\cA}{\mathcal A}
\newcommand{\cO}{\mathcal O}

\newcommand{\cS}{\mathcal S}

\newcommand{\cW}{\mathcal W}

\newcommand{\acts}{\curvearrowright}

\DeclareMathOperator{\dist}{dist}

\counterwithout{equation}{section}
\title[Free groups amenably act on unital simple AF-algebras]{Free groups amenably act on unital simple AF-algebras}
\author{Yuhei Suzuki}

\address{Yuhei Suzuki, Department of Mathematics, Faculty of Science, Hokkaido University,
Kita 10, Nishi 8, Kita-Ku, Sapporo, Hokkaido, 060-0810, Japan}
\email{yuhei@math.sci.hokudai.ac.jp}

\begin{document}

\begin{abstract}
We show that free groups admit amenable actions on unital simple AF-algebras.
This settles, for free groups, the last remaining case of the existence problem for amenable actions of non-amenable groups on classifiable simple \Cs-algebras.
The proof relies on the Baire category theorem.
This result was obtained with substantial assistance from GPT-6 Astra.
\end{abstract}
\maketitle
\section{Introduction}
The study of amenable actions on \Cs-algebras was pioneered by Anantharaman-Delaroche \cite{AD, ADT}
and has recently attracted renewed attention following the discovery of remarkable examples, most notably amenable actions (of non-amenable groups) on simple \Cs-algebras \cite{Suzeq}.
A highlight of these developments is the establishment of the dynamical Kirchberg--Phillips theorem by Gabe and Szab{\'o} \cite{GS1, GS2}.
In light of this achievement, it is natural and important to extend the dynamical classification theorem to a broader class of simple \Cs-dynamical systems.
Broadly speaking, classifiable simple \Cs-algebras fall into the following six classes.
\begin{enumerate}
\item Unital simple stably finite \Cs-algebras.
\item Unital simple purely infinite \Cs-algebras.
\item Non-unital simple stably finite \Cs-algebras $A$ with $\mathrm{K}_0(A)_+\neq \{0\}$.
\item Stable simple purely infinite \Cs-algebras.
\item Simple stably projectionless \Cs-algebras with a tracial state.
\item Simple stably projectionless \Cs-algebras without bounded tracial weights.
\end{enumerate}
We refer the reader to the introduction of \cite{CGS} and the survey article \cite{GLN} for the classification theory of simple \Cs-algebras.

While it was already known that amenable actions of non-amenable groups are abundant in four of these classes, namely (2)--(4) and (6)
(see \cite{Suzeq, OS, GS2, Suzsf, SuzP}), the existence of such actions in the remaining two classes has remained open for eight years since the problem was first raised in \cite{Suzeq}.
The difficulty is that, for algebras in these classes, it is hard to make an action amenable without destroying the required structural properties of the algebras.
Note that an example in class (1) yields one in (5) simply by tensoring it with a \Cs-algebra such as the Razak--Jacelon algebra $\cW$ \cite{J}.
In this paper, we resolve this problem for free groups,
thereby answering Problem D of \cite{GG2} in the affirmative.

The statement in the title follows from Theorem \ref{Thm:Fame} below, a new result in topological dynamics that may be of independent interest.
Since an amenable action of a non-amenable group admits no invariant probability measures,
our result can be seen as a refinement of Theorem 0.1 of \cite{FSZ}.
More precisely,
our amenability result can replace Lemma 3.5 and Proposition 3.6 of \cite{FSZ}, which form a key step in the proof of one of the main results (Theorem 0.1)
there and rely on a careful combinatorial construction of a subshift and a clever application of a ``$2$-to-$1$ instance of Hall's matching criterion''.

Let $\mathsf{P}$ be a property of topological dynamical systems (e.g.,  minimality, amenability).
Let $\Lambda < \Gamma$ be groups.
We say that $\alpha \colon \Gamma \acts X$ satisfies $\Lambda$-$\mathsf{P}$ if the restricted action $\alpha|_\Lambda$ satisfies $\mathsf{P}$.

Throughout this article, let $H$ be a countably infinite group.
In Section \ref{subsection:Phi}, we recall  the space $\Phi$ of subshifts of $H\times \IF_2$ over the Cantor set, which was introduced and studied in \cite{FSZ}.
\begin{Thm}\label{Thm:Fame}
$\IF_2$-amenability is a generic property in $\Phi$.
\end{Thm}
Combining this with Propositions 3.2 and 4.2 of \cite{FSZ}, we conclude:
\begin{Cor}\label{Cor:desired}
There exists a free, $\IF_2$-amenable, $H$-minimal action $\alpha \colon H\times \mathbb{F}_2 \acts X$ on a Cantor set $X$.
\end{Cor}
From Corollary \ref{Cor:desired}, we derive the statement in the title.
\begin{Thm}\label{Thm:ameAF}
There exists an amenable action of $\mathbb{F}_2$ on a unital simple separable AF-algebra.
\end{Thm}
\begin{proof}
Applying Corollary \ref{Cor:desired} to a countably infinite locally finite group $H$ (for example, $\bigoplus_{n\in \IN} \IZ/2\IZ$ or $\mathfrak{S}_{\infty}:=\bigcup_{n\in \IN} \mathfrak{S}_n$),
we obtain an action $\alpha \colon H \times \mathbb{F}_2 \acts X$ on a Cantor set
satisfying the three conditions therein.
Then, by the minimality and freeness of $\alpha|_H$,
the reduced crossed product $A:=C(X)\rca{\alpha} H$ is simple by \cite{KT}.
Since $X$ is a Cantor set and $H$ is locally finite and countable, $A$ is a unital separable AF-algebra.
Since $H$ is amenable and $\alpha$ is $\IF_2$-amenable, it follows that $\alpha$ is amenable.
Hence the reduced crossed product $A\rca{\alpha} \mathbb{F}_2=C(X)\rca{\alpha}(H\times \IF_2)$ is nuclear.
By Theorem 4.5 of \cite{AD}, the induced action $\alpha \colon \IF_2 \acts A$ is amenable.
\end{proof}
Note that the crossed product $A\rca{\alpha} \IF_2$ of the \Cs-dynamical system constructed in the proof of Theorem \ref{Thm:ameAF}
is a Kirchberg algebra by Theorem B of \cite{GG1}.

The proof of Theorem \ref{Thm:Fame}  is in the same spirit as the construction in \cite{Suzeq} (see \cite{SuzAd}, Corollary 2.4):
we prove density by encoding approximate witnesses of amenability into the tail coordinates of the Cantor set $\{0, 1\}^\IN$.
However, the present proof is more involved than that of \cite{Suzeq}, as the space $\Phi$ of actions considered here is more complex than the one used therein.

By amenability, the reduced and full crossed products of the dynamical system obtained in Theorem \ref{Thm:ameAF} coincide.
As a consequence, we obtain the following novel insight into full amalgamated free product \Cs-algebras.
\begin{Cor}\label{Cor:amal}
There are unital inclusions of unital simple stably finite \Cs-algebras $D \subset A_0, A_1$
whose full amalgamated free product $A_0 \mathop{\ast}\limits_{D} A_1$ is a Kirchberg algebra.
\end{Cor}
\begin{proof}
Let $\IF_2 \acts A$ be the action constructed in the proof of Theorem \ref{Thm:ameAF}.
Let $a, b$ denote the canonical generators of $\IF_2$.
Put $D:=A$, $A_0:= A\rca{\alpha} \langle a \rangle$, $A_1:=A\rca{\alpha}\langle b \rangle$.
The canonical inclusions $D \subset A_i$ ($i=0, 1$) give the desired pair.
\end{proof}
\begin{Rem}
Since $\IF_n$ embeds into $\IF_2$ for every $n\geq 3$, including $n=\infty$, Corollary \ref{Cor:desired} and Theorem \ref{Thm:ameAF} also hold for other free groups. In particular, Corollary \ref{Cor:amal} concerning full amalgamated free product decompositions holds in the $n$-factor case as well.
\end{Rem}
It would be interesting to determine which unital Kirchberg algebras admit such a full amalgamated free product decomposition.

\section{Preliminaries}
Here, we collect basic terminology used throughout this paper.
For background on \Cs-algebras and amenability of \Cs-dynamical systems, we refer the reader to \cite{BO} and \cite{OS}, respectively.

\subsection*{Conventions}
For sets $X$, $I$, $J$, we frequently make the canonical identifications among the Cartesian powers $(X^I)^J$, $(X^J)^I$, and $X^{I\times J}$ without explicit mention.
We regard $\IF_2$ as the subgroup $\{1_H\} \times \IF_2$ of $H\times \IF_2$. We include $0$ in $\IN$.
As in the proof of Corollary \ref{Cor:amal}, we write $a, b$ for the canonical generators of $\IF_2$, and set $S:=\{a^{\pm 1}, b^{\pm 1}\}$.
\subsection{The hyperspace of a Cantor set and the Vietoris topology}
Let $I$ be a countably infinite set.
On the set ${\rm CL}(\{0, 1\}^I)$ of non-empty closed subsets of the Cantor set $\{0, 1\}^I$,
the Vietoris topology can be described as follows.
For any finite subset $F\subset I$,
let $(p_F)_\ast \colon {\rm CL}(\{0, 1\}^I)\rightarrow \mathfrak{P}(\{0, 1\}^F)$
be the map induced by the coordinate projection $p_F$, where $\mathfrak{P}$ denotes the power set.
The Vietoris topology on ${\rm CL}(\{0, 1\}^I)$
is the weakest topology making the maps $(p_F)_\ast$ continuous,
where the finite sets $\mathfrak{P}(\{0, 1\}^F)$ are equipped with the discrete topology.
This topological space, called the \emph{hyperspace of $\{0, 1\}^I$}, is denoted by $\cH(\{0, 1\}^I)$.
It is easy to see that $\cH(\{0, 1\}^I)$ is homeomorphic to a closed subspace
of the product space $\prod_{F} \mathfrak{P}(\{0, 1\}^F)$, where $F$ ranges over finite subsets of $I$.
Hence it is compact and metrizable.

\subsection{Forest structure and metric on $H\times \IF_2$}
We regard $H \times \IF_2$ as a forest
by identifying it with the disjoint union of the copies $\{h\} \times \IF_2$ ($h\in H$) of the left Cayley tree $\IF_2$ defined by $S$.\footnote{Thus, the edge set on $H\times \IF_2$ is given by $\{((h, g), (h, sg)): h\in H, g\in \IF_2, s\in S\}$.}
Let $d$ be the graph metric on $H\times \IF_2$, which takes values in $\IN \cup \{\infty\}$.
For $n \in \IN$, we say that two subsets $B, C \subset H \times \IF_2$ are \emph{$n$-apart} if $\inf \{d(b, c) : b \in B, c \in C\} > n$.

With this setup, for $n \in \IN$, one can describe the set $\cB^*_n$ of subsets of $H\times \IF_2$ defined in Definition 4.1 of \cite{FSZ} as follows.
The set $\cB^*_n$ consists of all finite subsets $B$ of $H \times \IF_2$ such that any two distinct connected components of $B$ are $n$-apart.
Note that $\cB^\ast_n \supset \cB^\ast_{n+1}$.
\subsection{A space of subshifts $\Phi$} \label{subsection:Phi}
Following \cite{FSZ}, for a countably infinite group $\Lambda$ and a compact space $X$, we equip $X^\Lambda$ with the \emph{right} shift action: $g(x_h)_{h\in \Lambda}:=(x_{hg})_{h\in \Lambda}$.
We define $\Sub{(\{0, 1\}^{\IN})^{\Lambda}}\subset \cH({\{0, 1\}^{\IN\times \Lambda}})$
to be the set of all $\Lambda$-invariant closed non-empty subsets,
which is a closed subset.
Note that if $Y\in \Sub{(\{0, 1\}^{\IN})^{\Lambda}}$ is minimal and free,
then $Y$ is a Cantor set, since it has no isolated points.

For $n\in \IN$, write $\pi_n \colon (\{0, 1\}^{\IN})^{\Lambda}\rightarrow (\{0, 1\}^n)^\Lambda$
for the $\Lambda$-factor map induced from the coordinate projection $\{0, 1\}^\IN \rightarrow \{0, 1\}^n$.

We recall the space $\Phi \subset \Sub{(\{0, 1\}^{\IN})^{H \times \IF_2}}$ of subshifts of $H \times \IF_2$ over $\{0, 1\}^\IN$ introduced and studied in \cite{FSZ}.
First, for $n\in \IN$, a subshift $Y\in  \Sub{(\{0, 1\}^{\IN})^{ H \times \IF_2}}$ is said to be \emph{$\cB^\ast_n$-irreducible}
if for any $l\in \IN$, there is a finite subset $\mathcal{F} \subset H\times \IF_2$, called a \emph{witness} for $l$,  such that for any $k\in \IN$, any $y_1, \ldots, y_k\in Y$, and any $B_1, \ldots, B_k \in \cB^\ast_n$
satisfying $(\mathcal{F}\cdot B_i) \cap B_j =\emptyset$ for all distinct pairs $i$ and $j$, there exists $y\in Y$ with $\pi_l(y)|_{B_i}=\pi_l(y_i)|_{B_i}$ for all $i$.
For a more general definition, see Definition 3.1 of \cite{FSZ}.

Write $\cS_{\cB^\ast_n} \subset \Sub{(\{0, 1\}^{\IN})^{H \times \IF_2}}$ for the set of all $\cB^\ast_n$-irreducible subshifts.
We set $\Phi$ to be the closure of the (increasing) union $\bigcup_{n\in \IN}\cS_{\cB^\ast_n}$ in $\Sub{(\{0, 1\}^{\IN})^{ H \times \IF_2}}$.
Note that $\Phi$ is non-empty as it contains the full shift,
and is a compact metrizable space.
\section{Proof of Theorem \ref{Thm:Fame}}
The following result of \cite{FSZ} is fundamental in this paper.
\begin{Thm}[Propositions 3.2 and 4.2 in \cite{FSZ}; see also the paragraph preceding Proposition 4.2 therein]
For any countably infinite group $H$,
both freeness and $H$-minimality are dense $G_\delta$ properties in $\Phi$.
\end{Thm}

Let $\partial \IF_2$ denote the ideal boundary of $\IF_2$ (see Section 5.1 of \cite{BO}).
Let $K:=\Prob(\partial \IF_2) \subset C(\partial \IF_2)^\ast$, equipped with the weak-$\ast$ topology.
Write $\| \bullet \|$ for the dual norm of $C(\partial \IF_2)^\ast$.
Denote by $\bullet_\ast \colon \IF_2 \acts K$ the action induced by the canonical action $\IF_2 \acts \partial \IF_2$.

For each $\delta >0$, define the subset $Z_\delta \subset K^{\IF_2}$ to be
\[Z_\delta:=\{z\in K^{\IF_2}: \|z(sg)-s_\ast(z(g))\| \leq \delta\quad \mathrm{for~all~}s\in S, g\in \IF_2\}.\]
Note that for any $\mu \in K$, the orbit map $g\mapsto g_\ast \mu$ belongs to $Z_\delta$.
Hence $Z_\delta$ is non-empty. 
Since the norm is weak-$\ast$ lower semicontinuous, $Z_\delta$ is closed.
Clearly $Z_\delta$ is $\IF_2$-invariant with respect to the right shift action.
The key observation is that $Z_\delta$ is sufficiently irreducible while being close to amenable.
\begin{Lem}\label{Lem:irr}
Let $\delta>0$ and let $R\in \IN_{\geq 1}$ with $R\delta \geq 2$.
Set $r:=2R+1$. Let $C_1, \ldots, C_m \subset \IF_2$
be non-empty connected finite subsets which are pairwise $r$-apart.
Let $z_1, \ldots, z_m\in Z_\delta$.
Then there is $z\in Z_\delta$ with
\[z|_{C_i}=z_i|_{C_i}\quad \mathrm{for~all~}i.\]
\end{Lem}
\begin{proof}
The rough idea of the proof is to use the tree structure of $\IF_2$ to gradually deform the concatenation map $\bigsqcup_i z_i|_{C_i}$ into an orbit map outside $\bigsqcup_i C_i$.
Fix any $\mu \in K$.
For each $i$, let $\mathrm{pr}_i\colon \IF_2 \rightarrow C_i$ be the projection map,
that is, for $g\in \IF_2$, $\mathrm{pr}_i(g)\in C_i$ is defined to be the unique point nearest to $g$.

For each $i$,
define $\rho_i\colon \IF_2 \rightarrow \IR_{\geq 0}$
to be
\[\rho_i(g):=\frac{\dist(g, C_i)}{R}\quad \text{ for } g\in \IF_2.\]
We then define $z\in K^{\IF_2}$ by
\[
z(g) :=
\begin{cases}
(1 - \rho_i(g)) (g \cdot \mathrm{pr}_i(g)^{-1})_* (z_i(\mathrm{pr}_i(g))) + \rho_i(g) g_* \mu & \text{if } \rho_i(g) < 1 \text{ for some } i, \\
g_* \mu & \text{otherwise.}
\end{cases}
\]
Note that this is well-defined by the $r$-apartness of the $C_i$'s.
Clearly $z|_{C_i}=z_i|_{C_i}$ for all $i$.
Hence it suffices to show that $z\in Z_\delta$.
To check the condition, let $g\in \IF_2$ and $s\in S$.
If $\dist(g, C_i) \geq R+1$ for all $i$, then one has
$z(sg)=(sg)_\ast(\mu) = s_\ast(z(g))$.
Assume that $\dist(g, C_i) \leq R$ for some $i$.
Note that again by the $r$-apartness of the $C_j$'s,
such an $i$ is unique, say $i_0$, and for any other $j$, one has $\dist(sg, C_j) \geq R$.
When both $g$ and $sg$ are contained in $C_{i_0}$, we have $\|s_\ast(z(g))-z(sg)\| \leq \delta$ since $z_{i_0}\in Z_\delta$.
Otherwise, one has $|\dist(sg, C_{i_0})- \dist(g, C_{i_0})|\leq 1$ and $\mathrm{pr}_{i_0}(sg)=\mathrm{pr}_{i_0}(g)$,
hence $\|s_\ast(z(g))-z(sg)\|\leq 2/R \leq \delta$.
This proves $z\in Z_\delta$.
\end{proof}
Since $K$ is not totally disconnected, we next transfer $Z_\delta$
to a subshift over the Cantor set via a continuous surjection.
We fix a continuous surjection $q\colon \{0, 1\}^\IN \rightarrow K$, which exists since $K$ is a (non-empty) compact metrizable space.
Let $q^{\IF_2} \colon \{0, 1\}^{\IN \times \IF_2} \rightarrow K^{\IF_2}$ denote the surjection obtained by applying $q$ entrywise.
Then define
 \[W_\delta:=\left(q^{\IF_2}\right)^{-1}(Z_\delta)\in \Sub{(\{0, 1\}^{\IN})^{\IF_2}},\quad V_\delta:=(W_\delta)^H \in \Sub{(\{0, 1\}^{\IN})^{H\times \IF_2}}.\]
 Note that by Lemma \ref{Lem:irr} and the surjectivity of $q$,
 $W_\delta$ satisfies the same irreducibility condition as in Lemma \ref{Lem:irr} with the same $r\in \IN$ therein.
Now applying this irreducibility on each $\IF_2$-coset, we conclude the $\cB^\ast_r$-irreducibility of $V_\delta$, that is, $V_\delta \in \cS_{\cB^\ast_r} \subset \Phi$,
with the witness $\mathcal{F}=\{1_H\}\times \mathrm{Ball}(r)$ for all $l\in \IN$, where 
 \[\mathrm{Ball}(k):=\{g\in \IF_2: d(1_{\IF_2}, g)\leq k\} \quad \text{for $k\in \IN$}.\]

We next recall how to witness the amenability of $\IF_2 \acts \partial \IF_2$ from Section 5.1 of \cite{BO}.
 For $x\in \partial \IF_2$ and $j\in \IN$, let $x|_j$ denote the first $j$ letters of $x$, regarded as a reduced word in $\IF_2$ with $x|_0:=1_{\IF_2}$.
 For $N \in\IN_{\geq 1}$, define a continuous map $b_N \colon \partial \IF_2 \rightarrow \Prob(\IF_2)$
 to be 
 \[b_N(x):=\frac{1}{N}\sum_{j=0}^{N-1}\delta_{x|_j}.\]
 We equip $\Prob(\IF_2) \subset \ell^1(\IF_2)$ with the action $\bullet_\ast \colon \IF_2 \acts \Prob(\IF_2)$ induced by the left translations.
 Note that the image of $b_N$ is contained in $\Prob(\mathrm{Ball}(N-1))$.
It is not hard to see that
 \[\|b_N(sx)-s_\ast (b_N(x))\|_1 \leq 2/N\]
 for all $x\in \partial \IF_2$ and $s\in S$.
 
 Next, we extend each $b_N$ to an affine map on $K=\Prob(\partial\IF_2)$ by integration.
 More precisely, we define $P_N \colon K \rightarrow \Prob(\mathrm{Ball}(N-1))\subset \Prob(\IF_2)$ to be
 \[P_N(\mu)(g):=\int_{\partial \IF_2} b_N(x)(g) d\mu(x) \quad \text{for all $g\in \mathrm{Ball}(N-1)$ and $\mu \in K$}.\]
 Note that $P_N$ is well-defined and continuous because $\mathrm{Ball}(N-1)$ is a finite set.
By definition, one has the inequalities
 \begin{equation}\label{equation:Lip}
\|P_N(\mu)-P_N(\nu)\|_1 \leq \|\mu-\nu\|\quad \text{for all $\mu, \nu \in K$},
 \end{equation}
 
 \begin{equation}\label{equation:ame}
 \|P_N(s_\ast \mu) - s_\ast(P_N(\mu))\|_1 \leq \frac{2}{N} \quad \text{for all $s\in S$ and $\mu \in K$}.
 \end{equation}
From these conditions, for any $z\in Z_\delta$ and $s\in S$,  we obtain
\[\| P_N((sz)(1_{\IF_2}))-s_\ast P_N(z(1_{\IF_2})) \|_1 \leq \frac{2}{N}+\delta.\]
 Next, define $\theta \colon V_\delta \rightarrow K$ to be
 \[\theta(v):= q(v(1_{H\times \IF_2})) \quad \text{for $v\in V_\delta$}.\]
 Note that for any $v\in V_\delta$ and any $s\in S$,
 as $q^{\IF_2}(v|_{\IF_2}) \in Z_\delta$,
 one has
 \begin{equation}\label{equation:theta}
 \|\theta(sv)-s_\ast(\theta(v))\|\leq \delta.
\end{equation} 
 Let $\varsigma \colon \{0, 1\}^\IN \rightarrow \{0, 1\}^\IN$
 denote the backward shift: $\varsigma((x_n)_{n\in \IN}):=(x_{n+1})_{n\in \IN}$.
 For each $N\in \IN_{\geq 1}$,
 define $M_N \colon \{0, 1\}^{\IN \times H\times \IF_2} \rightarrow \Prob(\mathrm{Ball}(N-1)) \subset \Prob(\IF_2)$ to be
 $M_N(x):=P_N(q(\varsigma^N(x(1_{H\times \IF_2}))))$.
 For $n\in \IN_{\geq 1}$, we define
\[\cA_n := \bigcup_{k \in \IN_{\geq 1}} \left\{Y\in \Phi: \sup_{y\in Y} \| M_k(sy)-s_\ast(M_k(y))\|_1 <1/n \text{ for all $s\in S$} \right\}.\]
Since $\mathrm{Ball}(k-1)$ is a finite set,
the function $y\mapsto \| M_k(sy)-s_\ast(M_k(y))\|_1$ on  $\{0, 1\}^{\IN \times H\times \IF_2}$ is continuous.
Hence each $\cA_n$ is open in $\Phi$.
 
 Observe that any $Y\in \bigcap_{n\in \IN_{\geq 1}} \cA_n$ is $\IF_2$-amenable.
 Indeed, for any such $Y$,
 there exists a sequence $(k_n)_{n \in \IN_{\geq 1}} \subset \IN_{\geq 1}$
 such that the maps $({M_{k_n}}|_{Y})_{n\in \IN_{\geq 1}}$ witness the $\IF_2$-amenability of $Y$.
 Thus, by the Baire category theorem, Theorem \ref{Thm:Fame} reduces to the following claim.
 \begin{Lem}
 For any $N\in \IN_{\geq 1}$, the set $\cA_N$ is dense in $\Phi$.
 \end{Lem}
 \begin{proof}
 Let $\cO\subset \Phi$ be any non-empty open set.
 Pick $Y\in \cO \cap \cS_{\cB^\ast_n}$ for some $n\in \IN$.
 It suffices to show that $Y$ lies in the closure of $\cA_N$.
 
 Fix $\delta>0$ and $m_0 \geq N$ with $\delta + 2/m_0 <1/N$, and choose $R$ as in Lemma \ref{Lem:irr}, so that $V_\delta \in \cS_{\cB^\ast_{2R+1}}$. 
For $m\in \IN$, we identify $\{0, 1\}^\IN$ with $\{0, 1\}^m \times \{0, 1\}^\IN$ by the map $x \mapsto (p_m(x), \varsigma^{m}(x))$,
where $p_m \colon \{0, 1\}^\IN\rightarrow \{0, 1\}^m$ denotes the coordinate projection.

Set $Y_m:=\pi_m(Y) \times V_\delta$.
Since both $Y$ and $V_\delta$ belong to $\cS_{\cB^\ast_L}$ with $L:=\max\{n, 2R+1\}$,
we have $Y_m \in \cS_{\cB^\ast_L} \subset \Phi$ for all $m\in \IN$.
By definition, $\lim_{m\rightarrow \infty} Y_m = Y$.
Observe that for any $(y_0, v)\in \pi_m(Y) \times V_\delta=Y_m$, one has
$M_m(y_0, v)=P_m(\theta(v))$.
Thus, by (\ref{equation:Lip}), (\ref{equation:ame}) and (\ref{equation:theta}), we obtain $Y_m\in \cA_N$ for every $m\geq m_0$.
This completes the proof.
 \end{proof}

\subsection*{Acknowledgements}
The author is grateful to Professor Narutaka Ozawa for his kind guidance in using GPT-6 Astra (OpenAI).

\subsection*{Funding}
The author was supported by JSPS KAKENHI Grant-in-Aid for Scientific Research (C) 26K06820
and the Distinguished Researcher Grant of Hokkaido University.
\subsection*{AI statement}
GPT-6 Astra (OpenAI) proposed a proof, including the idea of using \cite{FSZ}, a key reference of which the author had previously been unaware.
AI tools were also used for language editing. 
The author reorganized the proof, wrote the exposition, verified all arguments, and takes full responsibility for the content.

\end{document}